\documentclass[a4paper]{article}

\usepackage{graphicx}
\usepackage{amsmath}
\usepackage{amssymb,verbatim}
\usepackage{mathrsfs}
\usepackage{dsfont}
\usepackage{url}
\usepackage[usenames,dvipsnames]{color}
\usepackage[compress]{cite}
\usepackage{mdwlist}
\usepackage{paralist}
\usepackage{algorithmic}
\usepackage{algorithm}
\usepackage{tikz}
\usepackage{tkz-graph}
\usepackage{array}
\usepackage{xcolor}
\usepackage{arydshln}

\usetikzlibrary{arrows,positioning,automata}

\newtheorem{lemma}{\sc Lemma}
\newtheorem{theorem}[lemma]{\sc Theorem}

\newtheorem{assumption}{\sc Assumption}
\newtheorem{definition}{\sc Definition}
\newtheorem{example}{\sc Example}

\renewcommand{\matrix}[2]{\left[\begin{array}{#1} #2 \end{array}\right] }

\DeclareMathOperator*{\diag}{diag}

\DeclareMathOperator*{\trace}{Tr}

\DeclareMathOperator*{\spanf}{span}

\renewcommand{\footnoterule}{%
  \kern 0pt
  \hrule width 0.3\textwidth height .5pt
  \kern 2pt
}

\floatname{algorithm}{Algorithm}

\newcommand{\IEEEQED}{~\rule[-1pt]{5pt}{5pt}\par\medskip}
\newenvironment{proof}{{\noindent \bf Proof:\ }}{ \hfill \IEEEQED}

\date{}

\begin{document}

\title{Optimal $H_\infty$ Control Design under Model Information Limitations \\ and State Measurement Constraints\thanks{The work was supported by the Swedish Research Council and the Knut and Alice Wallenberg Foundation. }}
\author{Farhad~Farokhi, Henrik~Sandberg,~and~Karl~H.~Johansson\thanks{The authors are with ACCESS Linnaeus Center, School of Electrical Engineering, KTH Royal Institute of Technology, SE-100 44 Stockholm, Sweden. E-mails: \{farakhi,hsan,kallej\}@kth.se }}

\maketitle

\begin{abstract} We present a suboptimal control design algorithm for a family of continuous-time parameter-dependent linear systems that are composed of interconnected subsystems. We are interested in designing the controller for each subsystem such that it only utilizes partial state measurements (characterized by a directed graph called the control graph) and limited model parameter information (characterized by the design graph). The algorithm is based on successive local minimizations and maximizations (using the subgradients) of the $H_\infty$--norm of the closed-loop transfer function with respect to the controller gains and the system parameters. We use a vehicle platooning example to illustrate the applicability of the results.
\end{abstract}

\section{Introduction}
Distributed and decentralized control design problem is a classical topic in the control literature (e.g., see~\cite{levine1971optimal,witsenhausen1968counterexample,sandell1978survey}).
Most of the available approaches in this field implicitly assume that the design procedure is done in a centralized fashion using the complete knowledge of the model parameters. However, this assumption is not realistic when dealing with large-scale systems due to several reasons. For instance, the overall system might be assembled from modules that are designed by separate entities without access to the entire set of model parameters because at the time of design this information was unavailable. Another reason could be that we want to keep the system maintenance simple by making it robust to nonlocal parameter changes; i.e., if a controller is designed knowing only local parameters, we do not need to redesign it whenever the parameters of a subsystem not in its immediate neighborhood change.  Financial gains, for instance, in the case of power network control, could also be a motivation for limited access to model knowledge  since competing companies are typically reluctant to share information on their production with each other. For a more detailed survey of the motivations behind control design using local model parameter information, see~\cite[Ch.\,1]{Farokhi-thesis2012}.

Recently, there have been some studies on control design with limited model information~\cite{Farokhi-thesis2012,langbort2010distributed,FLJ2012,farokhi2011dynamic}. For instance, the authors in~\cite{FLJ2012} introduce control design strategies as mappings from the set of plants to the set of structured static state-feedback controllers. They compare the control design strategies using a measure called the competitive ratio, which is defined to be the worst case ratio (over the set of all possible plants) of the closed-loop performance of the control design strategy in hand scaled by the best performance achievable having access to global model parameter information. Then, they seek a minimizer of the competitive ratio over a family of control design strategies that use only the parameters of their corresponding subsystems when designing controllers. Noting that, in those studies, the plants can vary over an unbounded set, the results are somewhat conservative. Additionally, all the aforementioned studies deal with discrete-time system as it was proved that the competitive ratio is unbounded when working with continuous-time systems~\cite{langbort2010distributed}. Not much have been done in optimal control design under limited model information for continuous-time systems.

In this paper, contrary to previous studies, we investigate continuous-time systems with parameters in a compact set. Specifically, we propose a numerical algorithm for calculating suboptimal $H_\infty$ control design strategies (i.e., mappings from the set of parameters to the set of structured static state-feedback controllers) for a set of parameter-dependent linear continuous-time systems composed of interconnected subsystems. We consider the case where each subsystem has access to a (possibly strict) subset of the system parameters when designing and implementing its local controller. Additionally, we assume that each local controller uses partial state measurements to close the feedback loop. To solve the problem, we first expand the control design strategies in terms of the system parameters (using a fixed set of basis functions) in such a way that each controller only uses its available parameters. Following the approach in~\cite{Apkarian1576856}, we calculate the subgradient of the $H_\infty$--norm of the closed-loop transfer function with respect to the controller gains and the system parameters. Then, we propose a numerical optimization algorithm based on successive local minimizations and maximizations of this performance measure with respect to the controller gains and the system parameters. Designing parameter-dependent controllers has a very rich history in the control literature, specially in gain scheduling and supervisory control; e.g., see~\cite{packard1994gain,rugh2000research,leith2000survey,nichols1993gain,morse1996supervisory,scorletti1998improved,apkarian2000parameterized}. However, most of these studies implicitly assume that the overall controller has access to all the parameters. Contrary to these studies, we assume that local controllers have access to only subsets of the system parameters.

The rest of the paper is organized as follows. In Section~\ref{sec:formulation}, we introduce the problem formulation. We propose a numerical algorithm for calculating a suboptimal $H_\infty$ control design strategy in Section~\ref{sec:results}. We illustrate the approach on a vehicle platooning example in Section~\ref{sec:platoon}. Finally, we present the conclusions in Section~\ref{sec:conclusion}.

\subsection{Notation}
Let the sets of integer and real numbers be denoted by $\mathbb{Z}$ and $\mathbb{R}$, respectively. Let $\mathbb{Z}_{> (\geq) n}=\{m\in\mathbb{Z}\;|\; m> (\geq) n\}$ and $\mathbb{R}_{> (\geq) x}=\{y\in\mathbb{R}\;|\; y> (\geq) x\}$ for $n\in\mathbb{Z}$ and $x\in\mathbb{R}$.

We use capital roman letters to denote matrices. The notation $A > (\geq) 0$ shows that the symmetric matrix $A$ is positive (semi-)definite. For any $q,m\in\mathbb{Z}_{\geq 1}$, we define the notation $\mathbb{B}_m^q=\{(Y_1,\dots,Y_q) \,|\,Y_i\in\mathbb{R}^{m\times m},Y_i\geq 0,\sum_{i=1}^q \trace(Y_i)=1\}$. We use $\mathbb{B}^q$ whenever the dimension $m$ is irrelevant (or can be deduced from the text). For any $A\in\mathbb{R}^{n\times m}$ and $B\in\mathbb{R}^{p \times q}$, we use $A\otimes B \in\mathbb{R}^{np \times mq}$ to denote the Kronecker product of these matrices.

Let an ordered set of real functions $(\xi_\ell)_{\ell=1}^L$ be given such that $\xi_\ell:\mathbb{R}^p\rightarrow \mathbb{R}$, $1\leq \ell\leq L$, are continuous functions with continuous first derivatives. We define $\spanf((\xi_\ell)_{\ell=1}^L)$ as the set composed of all linear combinations of the functions $(\xi_\ell)_{\ell=1}^L$; i.e., for any $f\in\spanf((\xi_\ell)_{\ell=1}^L)$, there exists at least one ordered set of real numbers $(x_\ell)_{\ell=1}^L$ such that $f(\alpha)=\sum_{\ell=1}^L x_\ell\xi_\ell(\alpha)$ for all $\alpha\in\mathbb{R}^p$. For any $n,m\in\mathbb{Z}_{\geq 1}$, $\spanf((\xi_\ell)_{\ell=1}^L)^{n\times m}$ denotes the set of all functions $A:\mathbb{R}^p\rightarrow \mathbb{R}^{n\times m}$ such that $A(\alpha)=\sum_{\ell=1}^L \xi_\ell(\alpha)A^{(\ell)}$ with $A^{(\ell)}\in\mathbb{R}^{n\times m}$ for all $1\leq \ell\leq L$.

We consider directed graphs with vertex set $\mathcal{V}=\{1,\dots,N\}$ for a fixed $N\in\mathbb{Z}_{\geq 1}$. For a graph $\mathcal{G}=(\mathcal{V},\mathcal{E})$, where $\mathcal{E}$ denotes its edge set, we define the adjacency matrix $S\in\{0,1\}^{N\times N}$ such that $s_{ij}=1$ if $(j,i) \in \mathcal{E}$, and $s_{ij}=0$ otherwise. We define the set of structured matrices $\mathcal{X}(S,(n_i)_{i=1}^N,(m_i)_{i=1}^N)$ as the set of all matrices $X\in\mathbb{R}^{n \times m}$ with $n=\sum_{i=1}^N n_i$ and $m=\sum_{i=1}^N m_i$ such that $X_{ij}=0 \in \mathbb{R}^{n_i \times n_j}$ whenever $s_{ij}=0$ for $1\leq i,j\leq N$.

For any function $f:\mathcal{U}\rightarrow \mathcal{Y}$, we call $\mathcal{U}$ the domain of $f$ and $\mathcal{Y}$ the codomain of $f$. Additionally, we define its image $f(\mathcal{U})$ as the set of all $y\in\mathcal{Y}$ such that $y=f(x)$ for a $x\in\mathcal{U}$.

For any $n\in\mathbb{Z}_{\geq 1}$, $I_n$ denotes the $n\times n$ identity matrix. To simplify the presentation, we use $I$ whenever the dimension can be inferred from the text. For any $n,m\in\mathbb{Z}_{\geq 1}$, we define $0_{n\times m}$ as the $n\times m$ zero matrix. Finally, let $\mathbf{1}_n\in\mathbb{R}^n$ be a vector of ones. 

\section{Mathematical Problem Formulation} \label{sec:formulation}
In this section, we introduce the underlying system model, the controller structure, and the closed-loop performance criterion.

\subsection{System Model}
Consider a continuous-time linear parameter-dependent system composed of $N\in\mathbb{Z}_{\geq 1}$ subsystems. Let subsystem~$i$, $1\leq i\leq N$, be described as
\begin{equation} \label{eqn:subsys:1}
\begin{split}
\dot{x}_i(t)&=\sum_{j=1}^N \big[ A_{ij}(\alpha_i)x_j(t)+ (B_w)_{ij}(\alpha_i)w_i(t)+(B_u)_{ij}(\alpha_i)u_i(t)\big],
\end{split}
\end{equation}
where $x_i(t)\in\mathbb{R}^{n_i}$ is the state vector, $w_i(t)\in\mathbb{R}^{m_{w,i}}$ is the exogenous input, $u_i(t)\in\mathbb{R}^{m_{u,i}}$ is the control input, and lastly, $\alpha_i\in\mathbb{R}^{p_i}$ is the parameter vector. Let us introduce the augmented state, control input, exogenous input, and parameter vector as
\begin{equation*}
\begin{split}
x(t)=&\matrix{ccc}{x_1(t)^\top & \cdots & x_N(t)^\top}^\top\in\mathbb{R}^n,\\
w(t)=&\matrix{ccc}{w_1(t)^\top & \cdots & w_N(t)^\top}^\top\in\mathbb{R}^{m_w},\\
u(t)=&\matrix{ccc}{u_1(t)^\top & \cdots & u_N(t)^\top}^\top\in\mathbb{R}^{m_u},\\ \alpha(t)=&\matrix{ccc}{\alpha_1(t)^\top & \cdots & \alpha_N(t)^\top}^\top\in\mathbb{R}^p,
\end{split}
\end{equation*}
where $n=\sum_{i=1}^N n_i$, $m_w=\sum_{i=1}^N m_{w,i}$, $m_u=\sum_{i=1}^N m_{u,i}$, and $p=\sum_{i=1}^N p_i$. This results in
$$
\dot{x}(t)=A(\alpha)x(t)+B_w(\alpha)w(t)+B_u(\alpha)u(t).
$$
We use the notation $\mathcal{A}$ to denote the set of all eligible parameter vectors~$\alpha$. We make the following standing assumption concerning the model matrices:
\begin{assumption} \label{assum:1} There exists a basis set $(\xi_{\ell})_{\ell=1}^{L}$ such that $A(\alpha)\in\spanf((\xi_{\ell})_{\ell=1}^{L})^{n\times n}$, $B_w(\alpha)\in\spanf((\xi_{\ell})_{\ell=1}^{L})^{n\times m_w}$, and $B_u(\alpha)\in\spanf((\xi_{\ell})_{\ell=1}^{L})^{n\times m_u}$.
\end{assumption}

\begin{example} \label{example:1}
Consider a parameter-dependent system described by
\begin{equation*}
\begin{split}
\dot{x}_1(t)&=(-2.0+\alpha_1)x_1(t)+ (0.1+0.4\sin(\alpha_1))x_2(t)+(0.6-0.3\sin(\alpha_1))u_1(t)+w_1(t),
\\ \dot{x}_2(t)&=+0.3x_1(t)+(-1.0-\alpha_2)x_2(t)+(1.0+0.1\cos(\alpha_2))
u_2(t)+w_2(t),
\end{split}
\end{equation*}
where $x_i(t)\in\mathbb{R}$, $u_i(t)\in\mathbb{R}$, $w_i(t)\in\mathbb{R}$, and $\alpha_i\in\mathbb{R}$ are respectively the state, the control input, the exogenous input, and the parameter of subsystem $i=1,2$. We define the set of eligible parameters as
$$
\mathcal{A}=\left\{\matrix{c}{\alpha_1 \\ \alpha_2}\in\mathbb{R}^2 \;\big| \; \alpha_i\in[-1,+1]  \mbox{ for } i=1,2\right\}.
$$
Clearly, this system satisfies Assumption~\ref{assum:1} with basis functions $\xi_1(\alpha)=1$, $\xi_2(\alpha)=\alpha_1$, $\xi_3(\alpha)=\sin(\alpha_1)$, $\xi_4(\alpha)=\cos(\alpha_2)$, and $\xi_5(\alpha)=\alpha_2$.
\hfill $\blacktriangleleft$
\end{example}

\subsection{Measurement Model and Controller}
Let a control graph $\mathcal{G}_{\mathcal{K}}$ with adjacency matrix $S_{\mathcal{K}}$ be given. We consider the case where each subsystem has access to a (potentially parameter-dependent) observation vector $y_i(t)\in\mathbb{R}^{o_{y,i}}$ that can be described by
$$
y_i(t)=\sum_{j=1}^N \big[(C_y)_{ij}(\alpha_i)x_j(t) +(D_{yw})_{ij}(\alpha_i)w_j(t) \big].
$$
Now, we can define the augmented observation vector as
$$
y(t)=\matrix{ccc}{y_1(t)^\top & \cdots & y_N(t)^\top}^\top\in\mathbb{R}^{o_y},
$$
where $o_y=\sum_{i=1}^N o_{y,i}$. Thus,
$$
y(t)=C_y(\alpha)x(t)+D_{yw}(\alpha)w(t).
$$
We say that the measurement vector $y(t)$ obeys the structure given by the control graph $\mathcal{G}_\mathcal{K}$ if $C_y(\mathcal{A})\in \linebreak[4] \mathcal{X}(S_{\mathcal{K}},(o_{y,i})_{i=1}^N,(n_i)_{i=1}^N)$ and $D_{yw}(\mathcal{A})\in \mathcal{X}(S_{\mathcal{K}},(o_{y,i})_{i=1}^N, (m_{w,i})_{i=1}^N)$, where the definition of the structured set $\mathcal{X}$ can be found in the notation subsection. We make the following standing assumption concerning the observation matrices:
\begin{assumption} \label{assum:2} For the same basis set $(\xi_{\ell})_{\ell=1}^{L}$ as in Assumption~\ref{assum:1}, $C_y(\alpha)\in\spanf((\xi_{\ell})_{\ell=1}^{L})^{o_y\times n}$ and $D_{yw}(\alpha)\in\spanf((\xi_{\ell})_{\ell=1}^{L})^{o_y\times m_w}$.
\end{assumption}

In this paper, we are interested in linear static state-feedback controllers of the form
\begin{equation} \label{eqn:controllaw}
u(k)=Ky(k),
\end{equation}
where $K\in\mathcal{K}=\mathcal{X}(I,(m_{u,i})_{i=1}^N,(o_{y,i})_{i=1}^N)$. Note that fol-lowing the same reasoning as in~\cite{Johnson1099586,Apkarian1576856}, the extension to fixed-order dynamic controllers is trivial (using just a change of variable).

\begin{figure}[t]
\centering
\vspace{-.2in}
$$
\hspace{-.1in}
\begin{array}{cc}
\begin{tikzpicture}[>=stealth',shorten >=2pt,initial/.style={}]
\node[state,minimum size=0.1cm,scale=0.7] (P1) {$P_1$};
\node[draw=none,fill=none,node distance=0.4cm] (dummy) [above =of P1] {};
\node[draw=none,fill=none,node distance=0.4cm] [right =of dummy] {$\mathcal{G}_\mathcal{K}$};
\node[state,minimum size=0.1cm,node distance=1.0cm,scale=0.7] (P2) [right =of P1]  {$P_2$};
\tikzset{mystyle/.style={->,double=black}}
\tikzset{mystyle/.style={->,relative=false,out=210,in=-20,double=black}}
\path (P2) edge [mystyle] (P1);
\path (P1) edge [->,in=180,out=90,looseness=8,double=black] (P1);
\path (P2) edge [->,in=0,out=90,looseness=8,double=black] (P2);
\end{tikzpicture}
&
\hspace{-.4in}
\begin{tikzpicture}[>=stealth',shorten >=2pt,initial/.style={}]
\node[state,minimum size=0.1cm,scale=0.7] (P1) {$P_1$};
\node[draw=none,fill=none,node distance=0.4cm] (dummy) [above =of P1] {};
\node[draw=none,fill=none,node distance=0.4cm] [right =of dummy] {$\mathcal{G}_\mathcal{C}$};
\node[state,minimum size=0.1cm,node distance=1.0cm,scale=0.7] (P2) [right =of P1]  {$P_2$};
\path (P1) edge [->,in=180,out=90,looseness=8,double=black] (P1);
\path (P2) edge [->,in=0,out=90,looseness=8,double=black] (P2);
\end{tikzpicture}
\end{array}
$$
\caption{\label{graph:example:GK} The control graph $\mathcal{G}_\mathcal{K}$ and the design graph $\mathcal{G}_\mathcal{C}$ utilized in the recurring numerical example.} 
\end{figure}
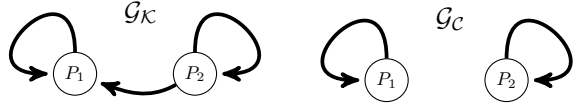

\setcounter{example}{0}
\begin{example}[Cont'd] \label{example:2} Let the control graph $\mathcal{G}_\mathcal{K}$ in Figure~\ref{graph:example:GK} represent the state-measurement availability. Consider the observation vectors
$$
y_1(t)=\matrix{c}{ x_1(t) \\ x_2(t) } \in\mathbb{R}^2, \hspace{.2in}
y_2(t)=x_2(t) \in\mathbb{R}.
$$
Clearly, the augmented observation vector obeys the structure dictated by $\mathcal{G}_\mathcal{K}$. Furthermore, since the measurement matrices are constant, they obviously satisfy Assumption~\ref{assum:2}. Finally, the controller~\eqref{eqn:controllaw} is given by
\begin{equation*}
\begin{split}
\matrix{c}{u_1(k)\\ u_2(k)}=\matrix{cc}{K_{11} & 0 \\ 0  & K_{22}} \matrix{c}{y_1(k)\\ y_2(k)},
\end{split}
\end{equation*}
where $K_{11}\in\mathbb{R}^{1\times 2}$ and $K_{22}\in\mathbb{R}$.
\hfill $\blacktriangleleft$
\end{example}

\subsection{Control Design Strategy}
Following~\cite{FLJ2012}, we define a control design strategy $\Gamma$ as a mapping from $\mathcal{A}$ to $\mathcal{K}$. Let a control design strategy $\Gamma:\mathcal{A}\rightarrow \mathcal{K}$ be partitioned following the measurement vector and the control input dimensions as
\begin{equation*}
\Gamma=\matrix{ccc}{ \Gamma_{11} & \cdots & 0 \\ \vdots & \ddots & \vdots \\ 0 & \cdots & \Gamma_{NN} },
\end{equation*}
where each block $\Gamma_{ii}$ represents a map $\mathcal{A} \rightarrow \mathbb{R}^{m_{u,i} \times o_{y,i}}$. Let a directed graph $\mathcal{G}_{\mathcal{C}}$ with adjacency matrix $S_{\mathcal{C}}$ be given. We say that the control design strategy $\Gamma$ has structure $\mathcal{G}_{\mathcal{C}}$ if $\Gamma_{ii}$, $1\leq i\leq N$, is only a function of $\left\{ \alpha_j \; | \; (s_{\mathcal{C}})_{ij} \neq 0 \right \}$. Let $\mathcal{C}$ denote the set of all control design strategies $\Gamma$ with structure $\mathcal{G}_{\mathcal{C}}$. We make the following standing assumption:
\begin{assumption} There exists a basis set $(\eta_{\ell'})_{\ell'=1}^{L'}$ such that $\Gamma\in\spanf((\eta_{\ell'})_{\ell'=1}^{L'})^{m_u\times o_y}$.
\end{assumption}

Now, we define $\mathcal{C}((\eta_{\ell'})_{\ell'=1}^{L'})=\mathcal{C}\cap \spanf((\eta_{\ell'})_{\ell'=1}^{L'})^{m_u\times o_y}$ as the set of all control design strategies over which we optimize the closed-loop performance.

\setcounter{example}{0}
\begin{example}[Cont'd] The design graph $\mathcal{G}_{\mathcal{C}}$ in Figure~\ref{graph:example:GK} illustrates the available plant model information. We use the basis functions $\eta_1(\alpha)=1$, $\eta_2(\alpha)=\alpha_1$, $\eta_3(\alpha)=\alpha_1^2$, and $\eta_4(\alpha)=\alpha_2$ for parameterizing the control design strategies. Clearly, any $\Gamma\in\mathcal{C}(\{\eta_{\ell'}\}_{\ell'=1}^4)$ can be expressed in the form
$$
\Gamma(\alpha)=\sum_{\ell'=1}^4 G^{(\ell')}\eta_{\ell'}(\alpha),
$$
with
\begin{equation*}
\begin{split}
G^{(1)}&=\matrix{ccc}{* & * & 0 \\ 0 & 0 & *}, \hspace{.2in}
G^{(2)}=\matrix{ccc}{* & * & 0 \\ 0 & 0 & 0},\\
G^{(3)}&=\matrix{ccc}{* & * & 0 \\ 0 & 0 & 0}, \hspace{.2in}
G^{(4)}=\matrix{ccc}{0 & 0 & 0 \\ 0 & 0 & *},
\end{split}
\end{equation*}
where * denotes the nonzero entries of these matrices. Note that the functions $\{\eta_{\ell'}\}_{\ell'=1}^4$ are indeed design choices and we can improve the closed-loop performance by increasing the number of the basis functions. However, this can only be achieved at the price of a higher computational time.
\hfill $\blacktriangleleft$
\end{example}

\subsection{Performance Metric}
Let us introduce the performance measure output vector
\begin{equation} \label{eqn:z(t)}
z(t)=C_zx(t)+D_{zw}w(t)+D_{zu}u(t)\in\mathbb{R}^{o_{z}}.
\end{equation}
We are interested in finding a control design method $\Gamma$ that solves the optimization problem
\begin{equation} \label{eqn:optimization}
\min_{\Gamma\in\mathcal{C}((\eta_{\ell'})_{\ell'=1}^{L'})} \max_{\alpha\in\mathcal{A}} \;\; \left\|T_{zw}\left(s;\Gamma,\alpha\right)\right\|_\infty,
\end{equation}
where $T_{zw}(s;\Gamma,\alpha)$ denotes the closed-loop transfer function from the exogenous input $w(t)$ to the performance measurement vector $z(t)$ for $\alpha\in\mathcal{A}$. We make the following assumptions to avoid singularities in the optimal control problem:

\begin{assumption} \label{assum:4} $D_{zu}^\top D_{zu}=I$ and $D_{yw} D_{yw}^\top=I$.
\end{assumption}

These assumptions are common in the $H_\infty$-control design literature~\cite[p.\,288]{zhou1998essentials}. However, notice that the conditions in Assumption~\ref{assum:4} are only sufficient (and not necessary). For instance, although $D_{yw}=0$ in Example~\ref{example:1}, as we will see later, a nontrivial solution indeed exists and the optimal control problem is in fact well-posed.

To simplify the presentation in what follows, we define the notation
$$
J(\Gamma,\alpha)=\left\|T_{zw}\left(s;\Gamma,\alpha\right)\right\|_\infty.
$$
Now, noting that there may exist many local solutions to the optimization problem~(\ref{eqn:optimization}), it is difficult to find the global solution of this problem. Hence, we define:
\begin{definition} A pair $(\Gamma^*,\alpha^*)\in\mathcal{C}((\eta_{\ell'})_{\ell'=1}^{L'}) \times \mathcal{A}$ is a saddle point of $J:\mathcal{C}((\eta_{\ell'})_{\ell'=1}^{L'})\times \alpha\rightarrow \mathbb{R}_{\geq 0}$ if there exists a constant $\epsilon\in\mathbb{R}_{>0}$ such that
$$
J(\Gamma^*,\alpha)\leq J(\Gamma^*,\alpha^*) \leq J(\Gamma,\alpha^*),
$$
for any $(\Gamma,\alpha)\in\mathcal{C}((\eta_{\ell'})_{\ell'=1}^{L'}) \times \mathcal{A}$ where $\|\Gamma-\Gamma^*\|\leq \epsilon$ and $\|\alpha-\alpha^*\|\leq \epsilon$.
\end{definition}

Evidently, the global solution of the minimax optimization problem~(\ref{eqn:optimization}) is also a saddle point of $J$. However, there might be many more saddle points. In the rest of this paper, we focus on finding a saddle point $(\Gamma^*,\alpha^*)$ of $J$. To make sure that the set of saddle points is nonempty, we make the following standing assumption:
\begin{assumption} \label{assum:5} The set of all eligible parameters $\mathcal{A}$ is a compact subset of $\mathbb{R}^p$. In addition, for any $\alpha\in\mathcal{A}$, the pair $(A(\alpha),B_u(\alpha))$ is stabilizable and the pair $(A(\alpha),C_y(\alpha))$ is detectable.
\end{assumption}
Notice that Assumption~\ref{assum:5} is only a necessary condition for the existence of a saddle point solution since we are solving a decentralized control design problem rather than a centralized one. If we switch the stabilizability and the detectability conditions with the absence of unstable fixed modes, this assumption becomes more realistic (but still not sufficient because of the asymmetric parameter dependencies).

\setcounter{example}{0}
\begin{example}[Cont'd] In this example, we are interested in minimizing the closed-loop transfer function from the exogenous inputs to the performance measurement vector with $C_z=[I_2\;\;0_{2\times 2}]^\top$, $D_{zu}=[0_{2\times 2}\;\;I_2]^\top$, and $D_{zw}=0$. Clearly, the choice of $D_{zu}$ satisfies Assumption~\ref{assum:4}. It is easy to check that the system satisfies Assumption~\ref{assum:5} as well.
\hfill $\blacktriangleleft$
\end{example}

\setcounter{footnote}{2}
\renewcommand{\thefootnote}{\fnsymbol{footnote}}

\section{Optimization Algorithm} \label{sec:results}
In this section, we develop a numerical algorithm for finding a saddle point $(\Gamma^*,\alpha^*)$ of $J$.
We start by calculating subgradients\footnote{We say that a vector $g\in\mathcal{X}$ is a subgradient of $f:\mathcal{X}\rightarrow \mathbb{R}$ at $x\in\mathcal{X}$ if for all $x'\in\mathcal{X}$, $f(x')\geq f(x)+g^\top(x'-x)$. Let $\partial f(x)$ denote the set of subgradients of $f$ at the point $x\in\mathcal{X}$. If $f$ is convex, then $\partial f(x)$ is nonempty and bounded. We would like refer interested readers to~\cite{shor1985minimization,boyd2003subgradient} (and the references therein) for a detailed review of the subgradients and numerical optimization algorithm using them.} $\Delta \Gamma\in\partial_{\Gamma}J(\Gamma,\alpha)$ and $\Delta \alpha \in \partial_{\alpha}J(\Gamma,\alpha)$ for any $(\Gamma,\alpha)\in\mathcal{C}((\eta_{\ell'})_{\ell'=1}^{L'}) \times \mathcal{A}$.

\begin{lemma} \label{lemma:1} Let us define the transfer functions in
\begin{equation} \label{eqn:tfneeded:1}
\begin{split}
&\matrix{cc}{\hspace{-.05in}T_{zw}(s;\Gamma,\alpha)\hspace{-.05in} & \hspace{-.05in}G_{12}(s;\Gamma,\alpha) \\ G_{21}(s;\Gamma,\alpha)\hspace{-.05in} & \bullet } = \matrix{c}{ C'_{\mathrm{cl}}(\Gamma,\alpha) \\C_{y'}(\alpha)}\left(sI-A'_{\mathrm{cl}}(\Gamma,\alpha)\right)^{-1} \matrix{cc}{B'_{\mathrm{cl}}(\Gamma,\alpha) & B_u(\alpha)}+\matrix{cc}{\hspace{-.05in}D'_{\mathrm{cl}}(\Gamma,\alpha) \hspace{-.05in} &\hspace{-.05in}D_{zu}\hspace{-.05in} \\ \hspace{-.05in}D_{y'w}(\alpha) & \bullet},
\end{split}
\end{equation}
with
\begin{equation*}
\begin{split}
A'_{\mathrm{cl}}(\Gamma,\alpha)&=A(\alpha) +B_u(\alpha)K'C_{y'}(\alpha), \\
B'_{\mathrm{cl}}(\Gamma,\alpha)&=B_w(\alpha) +B_u(\alpha)K'D_{y'w}(\alpha)\\
C'_{\mathrm{cl}}(\Gamma,\alpha)&=C_z(\alpha) +D_{zu}(\alpha)K'C_{y'}(\alpha), \\
D'_{\mathrm{cl}}(\Gamma,\alpha)&=D_{zw}(\alpha) +D_{zu}(\alpha)K'D_{y'w}(\alpha),
\end{split}
\end{equation*}
where $K'=[G^{(1)}\;\cdots\;G^{(L')}]$ and
$$
C_{y'}(\alpha)\hspace{-.05in}=\hspace{-.05in}\matrix{c}{
\hspace{-.05in}\eta_1(\alpha) C_y(\alpha) \hspace{-.05in}\\ \vdots \\ \hspace{-.05in}\eta_{L'}(\alpha) C_y(\alpha)\hspace{-.05in}}\hspace{-.05in},\;\;
D_{y'w}(\alpha)\hspace{-.05in}=\hspace{-.05in}\matrix{c}{ \hspace{-.05in}\eta_1(\alpha) D_{yw}(\alpha)\hspace{-.05in} \\ \vdots \\ \hspace{-.05in}\eta_{L'}(\alpha) D_{yw}(\alpha)\hspace{-.05in}}\hspace{-.05in}.
$$
Furthermore, let $\Delta \Gamma=\sum_{\ell'=1}^{L'} \Delta G^{(\ell')}$ be such that $\Delta G^{(\ell')}\in\mathbb{R}^{m\times o_y}$ are defined in
\begin{equation} \label{eqn:DeltaG}
\begin{split}
&\matrix{ccc}{\Delta G^{(1)} & \cdots & \Delta G^{(L')} }=\left\|T_{zw}\left(s;\Gamma,\alpha \right)\right\|_\infty^{-1}\sum_{\nu=1}^q \mathrm{Re}\left\{G_{21}(j\omega_\nu;\Gamma,\alpha)T_{zw} (j\omega_\nu;\Gamma,\alpha)^*Q_\nu Y_\nu Q_\nu^* G_{12}(j\omega_\nu;\Gamma,\alpha)\right\}^\top\hspace{-.1in},
\end{split}
\end{equation}
where $\|T_{zw}(s;\Gamma,\alpha)\|_\infty$ is attained at a finite number of frequencies $(\omega_1,\dots,\omega_q)$ and $(Y_1,\dots,Y_q)\in\mathbb{B}^q$. In addition, the columns of $Q_\nu$, $1\leq \nu\leq q$, are chosen so as to form an orthonormal basis for the eigenspace of $T_{zw}(j\omega_\nu;\Gamma,\alpha) T_{zw}(j\omega_\nu;\Gamma,\alpha)^*$ associated with the leading eigenvalue $\|T_{zw}(s;\Gamma,\alpha)\|_\infty$. Then, $\Delta \Gamma\in\partial_{\Gamma}J(\Gamma,\alpha)$.
\end{lemma}

\begin{proof} Due to space constraints, we only present a sketch of the proof here. First, we prove that the closed-loop system
\begin{equation*}
\left\{\begin{array}{rl}\dot{x}(t)&\hspace{-.1in}=A(\alpha)x(t)+ B_w(\alpha)w(t)+B_u(\alpha)u(t),
\\ z(t)&\hspace{-.1in}=C_zx(t)+D_{zw}w(t)+D_{zu}u(t), \\ y'(t)&\hspace{-.1in}=C_{y'}(\alpha)x(t)+D_{y'w}(\alpha)w(t),\\ u(t)&\hspace{-.1in}=K'y'(t), \end{array} \right.
\end{equation*}
is equivalent to the closed-loop system that we introduced in the previous section. Then, we can use the method presented in~\cite{Apkarian1576856} for calculating the subgradients of the closed-loop performance with respect to the controller gain. Doing so, we find $\Delta G^{(\ell')}\in\partial_{G^{\ell'}}J(\Gamma,\alpha)$ for $1\leq \ell'\leq L'$. Finally, we get $\sum_{\ell'=1}^{L'} \Delta G^{(\ell')}\eta_{\ell'}\in\partial_{\Gamma}J(\Gamma,\alpha)$.
\end{proof}

\begin{lemma} Let us define the transfer functions in
\begin{equation} \label{eqn:tfneeded:2}
\begin{split}
\matrix{cc}{T_{zw}(s;\Gamma,\alpha) & H_{12}(s;\Gamma,\alpha) \\ H_{21}(s;\Gamma,\alpha) & \bullet }= \matrix{c}{ C''_{\mathrm{cl}}(\Gamma,\alpha) \\C_{y''}}\left(sI-A''_{\mathrm{cl}}(\Gamma,\alpha)\right)^{-1} \matrix{cc}{B''_{\mathrm{cl}}(\Gamma,\alpha) & B_{u''}}+\matrix{cc}{D''_{\mathrm{cl}}(\Gamma,\alpha) & D_{zu''} \\ D_{y''w} & \bullet},
\end{split}
\end{equation}
with
\begin{equation*}
\begin{split}
A''_{\mathrm{cl}}(\Gamma,\alpha)&=B_{u''}K''(\alpha)C_{y''}, \\
B''_{\mathrm{cl}}(\Gamma,\alpha)&=B_{u''}K''(\alpha)D_{y''w}, \\
C''_{\mathrm{cl}}(\Gamma,\alpha)&=C_z+D_{zu''}K''(\alpha)C_{y''}, \\
D''_{\mathrm{cl}}(\Gamma,\alpha)&=D_{zw}+D_{zu''}K''(\alpha)D_{y''w},
\end{split}
\end{equation*}
where
$$
C_{y''}=\matrix{c}{A^{(1)}\\ \vdots \\ A^{(L)} \\ \mathbf{1}_{L+1}\otimes \matrix{c}{G^{(1)} C_y^{(1)} \\ G^{(1)} C_y^{(2)} \\ \vdots \\ G^{(1)} C_y^{(L)} \\ \vdots \\ G^{(L')} C_y^{(L)}} \\ 0_{(nL+m_uL(L+1)L')\times n}}\hspace{-.04in},
\hspace{.5in}
D_{y''w}=\matrix{c}{0_{(nL+m_uL(L+1)L')\times m_w} \\ B_w^{(1)}\\ \vdots \\ B_w^{(L)} \\ \mathbf{1}_{L+1}\otimes \matrix{c}{G^{(1)} D_{yw}^{(1)} \\ G^{(1)} D_{yw}^{(2)} \\ \vdots \\ G^{(1)} D_{yw}^{(L)}\\ \vdots \\ G^{(L')} D_{yw}^{(L)} }}\hspace{-.04in},
$$
$$
D_{zu''}=\matrix{c}{0_{(nL+m_uL^2L') \times o_z} \\ \mathbf{1}_{LL'} \otimes D_{zu}^\top \\ 0_{(nL+m_uL^2L') \times o_z} \\ \mathbf{1}_{LL'} \otimes D_{zu}^\top}^\top,
$$
and
$$
B_{u''}=\big[\mathbf{1}_{L}^\top \otimes I_{n\times n} \; \Upsilon \;\; 0_{n\times nLL'} \; \mathbf{1}_{L}^\top \otimes I_{n\times n} \; \Upsilon \;\; 0_{n\times nLL'}\big]^\top,
$$
with
$$
\Upsilon=\big[\mathbf{1}_{LL'}^\top \otimes B_{u}^{(1)\top} \;\; \cdots \;\; \mathbf{1}_{LL'}^\top \otimes B_{u}^{(L)\top}\big].
$$
Additionally, we have
\begin{equation*}
\begin{split}
K''(\alpha)=\diag(&\Xi(\alpha)\otimes I_{n}, \,\Xi(\alpha)\otimes\Psi(\alpha)\otimes\Xi(\alpha)\otimes I_{m_u}, \\ & \Psi(\alpha)\otimes\Xi(\alpha)\otimes I_{m_u},\,  \Xi(\alpha)\otimes I_{n},  \Xi(\alpha)\otimes\Psi(\alpha)\otimes\Xi(\alpha)\otimes I_{m_u}, \Psi(\alpha)\otimes\Xi(\alpha)\otimes I_{m_u}).
\end{split}
\end{equation*}
where, for all $\alpha\in\mathbb{R}^p$, $\Xi(\alpha)=\diag(\xi_{1}(\alpha),\dots,\xi_{L}(\alpha))$ and $\Psi(\alpha)=\diag(\eta_{1}(\alpha),\dots,\eta_{L'}(\alpha))$.
Furthermore, let $\Delta \alpha=[\Delta \alpha_1 \; \cdots \; \Delta \alpha_p]^\top$ be such that the scalars $\Delta \alpha_i\in\mathbb{R}$, $1\leq i\leq p$, are calculated using
\begin{equation} \label{eqn:Deltaalpha}
\begin{split}
\Delta \alpha_i=&
\left\|T_{zw}\left(s;\Gamma,\alpha\right)\right\|_\infty^{-1} \sum_{\nu=1}^q \mathrm{Re}\left\{ \trace \left[H_{21}(j\omega_\nu;\Gamma,\alpha)T_{zw}(j\omega_\nu;\Gamma,\alpha)^*Q_\nu Y_\nu Q_\nu^* H_{12}(j\omega_\nu;\Gamma,\alpha) \frac{\partial}{\partial \alpha_i} K''(\alpha) \right] \right\},
\end{split}
\end{equation}
where $\|T_{zw}(s;\Gamma,\alpha)\|_\infty$ is attained at a finite number of frequencies $(\omega_1,\dots,\omega_q)$ and $(Y_1,\dots,Y_q)\in\mathbb{B}^q$. In addition, the columns of $Q_\nu$, $1\leq \nu\leq q$, form an orthonormal basis of the eigenspace of $T_{zw}(j\omega_\nu;\Gamma,\alpha) T_{zw}(j\omega_\nu;\Gamma,\alpha)^*$ associated with the leading eigenvalue $\|T_{zw}(s;\Gamma,\alpha)\|_\infty$. Then, $\Delta \alpha \in \partial_{\alpha}J(\Gamma,\alpha)$.
\end{lemma}

\begin{proof} The proof follows the same line of reasoning as in the proof of Lemma~\ref{lemma:1}.
\end{proof}

\begin{algorithm}[t]
\caption{\label{alg:2} A numerical algorithm for calculating a saddle point $(\Gamma^*,\alpha^*)$ of $J$.}
\begin{algorithmic}[1]
\REQUIRE $\{G^{(\ell')}(0)\}_{\ell'=0}^{L'}$, $\alpha(0)$ , $\epsilon,\varepsilon\in\mathbb{R}_{>0}$, $\{\mu_k\}_{k=1}^\infty$
\ENSURE $\Gamma^*$, $\alpha^*$
\STATE $k\leftarrow 0$
\REPEAT
\STATE $\Gamma^{(k)}\leftarrow\sum_{\ell'=1}^{L'} G^{(\ell')}(k)\eta_{\ell'}$
\STATE $\bar{\alpha}(0)\leftarrow \alpha(k)$
\STATE $\tau\leftarrow 0$
\REPEAT
\STATE $\bar{\alpha}(\tau+1)\leftarrow P_\mathcal{A}(\bar{\alpha}(\tau) + \mu_{\tau} g_{k,\tau})$ where $ g_{k,\tau}\in \partial_{\alpha}J(\Gamma^{(k)},\alpha)$ calculated at $\bar{\alpha}(\tau)$ and $P_\mathcal{A}(\cdot)$ is the projection to $\mathcal{A}$
\STATE $\tau\leftarrow \tau+1$
\UNTIL{$|J(\Gamma^{(k)},\bar{\alpha}(\tau))- J(\Gamma^{(k)},\bar{\alpha}(\tau-1))|\leq \varepsilon$}
\STATE $\alpha(k+1)\leftarrow \bar{\alpha}(\tau)$
\FOR{$\ell'=1,\dots,L'$}
\STATE $G^{(\ell')}(k+1)\leftarrow P_\mathcal{C}(G^{(\ell')}(k) - \mu_k \Delta G^{(\ell')}(k))$ where $\Delta G^{(\ell')}(k)\in\partial_{G^{\ell'}}J(\Gamma,\alpha(k+1))$ calculated at $\Gamma^{(k)}$ and $P_\mathcal{C}(\cdot)$ is the projection to $\mathcal{C}((\eta_{\ell'})_{\ell'=1}^{L'})$
\ENDFOR
\STATE $k\leftarrow k+1$
\UNTIL{$|J(\Gamma^{(k-1)},\alpha(k-1))-J(\Gamma^{(k)},\alpha(k))|\leq \epsilon$}
\STATE $\Gamma^*\leftarrow \sum_{\ell'=1}^{L'} G^{(\ell')}(k)\eta_{\ell'}$
\STATE $\alpha^*\leftarrow \alpha(k)$
\end{algorithmic}
\end{algorithm}

Algorithm~\ref{alg:2} introduces a numerical algorithm for finding a saddle point of $J$, or equivalently, a local solution of the optimization problem in~(\ref{eqn:optimization}).

\begin{theorem} In Algorithm~\ref{alg:2}, let $\{\mu_k\}_{k=0}^\infty$ be chosen such that $ \lim_{k\rightarrow \infty}\sum_{z=1}^k \mu_z=\infty$ and $\lim_{k\rightarrow \infty}\sum_{z=1}^k \mu_z^2<\infty$. Assume that there exists $C\in\mathbb{R}$ such that $\left\|g_{k,\tau}\right\|_2\leq C$ and $\|\Delta G^{(\ell')}(k)\|_2\leq C$ for all $k,\tau\geq 0$ and $1\leq \ell' \leq L'$. Then, if $\lim_{k\rightarrow \infty}(\Gamma^{(k)},\alpha(k))$ exists, it is a saddle point of $J$.
\end{theorem}

\begin{proof} The proof follows from the convergence properties of conventional subgradient optimization algorithms~\cite{boyd2003subgradient}. \end{proof}

\begin{figure*}[t!]
\centering
$$
\begin{array}{cp{.4in}cp{.4in}c}
\includegraphics[width=0.35\linewidth]{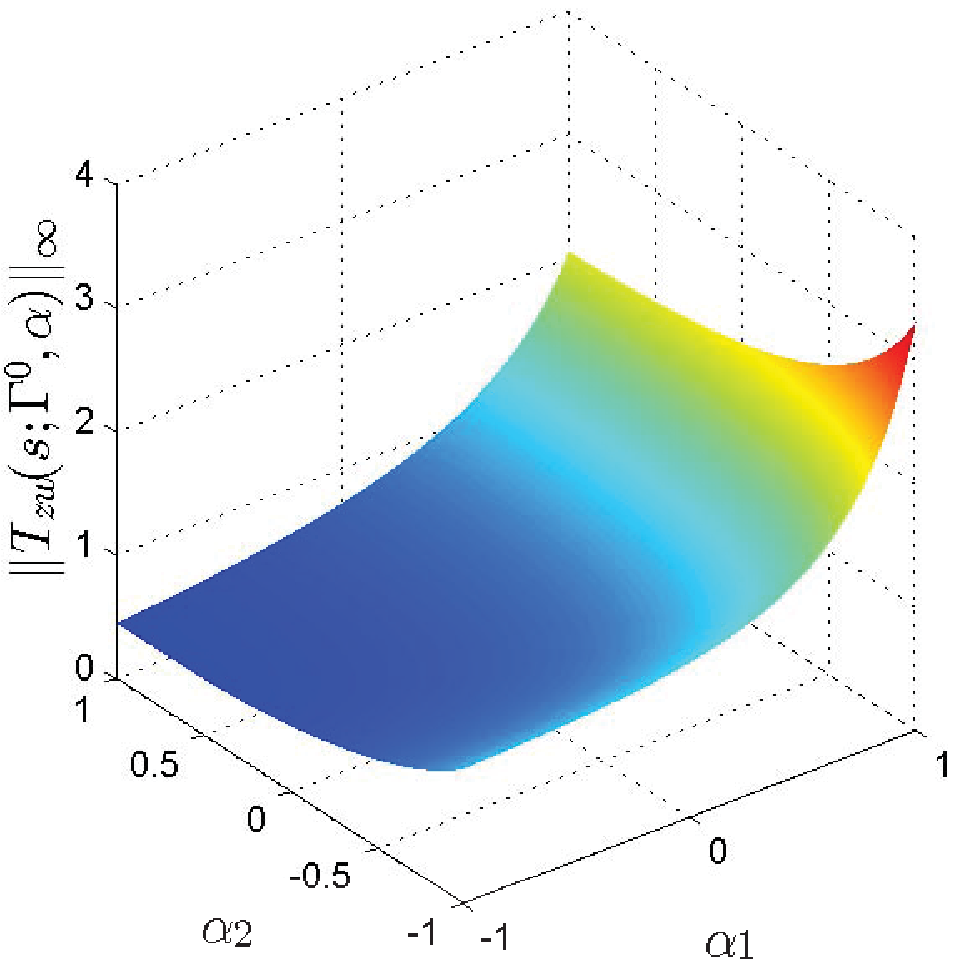}&\hspace{.9in}&
\includegraphics[width=0.35\linewidth]{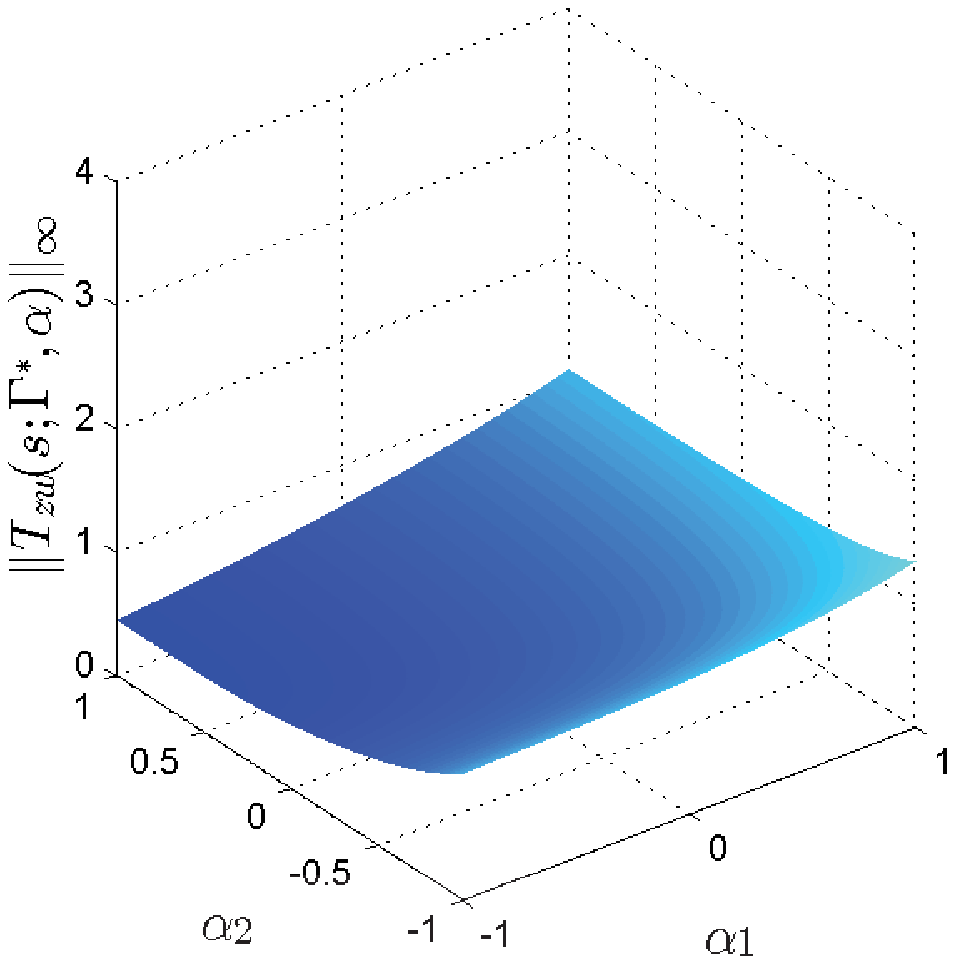}&\hspace{.9in}&
\includegraphics[width=0.06\linewidth]{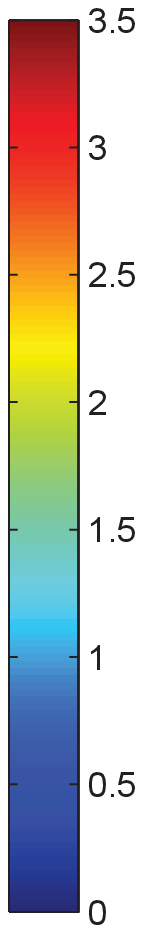}
\end{array}
$$
\caption{\label{figure1} The initial closed-loop performance $\|T_{zw}(s;\Gamma^0,\alpha)\|_\infty$ (left) and the optimal closed-loop performance $\|T_{zw}(s;\Gamma^*,\alpha)\|_\infty$ (right) as function of the parameters $\alpha_i$, $i=1,2$.}
\vspace{-.2in}
\end{figure*}

\setcounter{example}{0}
\begin{example}[Cont'd] Let us initialize Algorithm~\ref{alg:2} at $\alpha(0)=[0.0 \; -0.0 ]^\top$ and
$$
\Gamma^0(\alpha)=\matrix{
>{\centering\arraybackslash$} p{1.3cm} <{$} >{\centering\arraybackslash$} p{1.3cm} <{$} >{\centering\arraybackslash$} p{1.3cm} <{$}
}{+0.0 & +0.0 & +0.0 \\ +0.0 & +0.0 & -0.5}.
$$
Furthermore, we pick $\epsilon=\varepsilon=10^{-3}$ and $\mu_k=0.1/k$ for all $k\in\mathbb{Z}_{\geq 1}$. This results in
$\Gamma^*(\alpha)=G^{(1)}+G^{(2)}\alpha_1+G^{(3)}\alpha_1^2 +G^{(4)}\alpha_2,$
where
\begin{equation*}
\begin{split}
G^{(1)}&=\matrix{
>{\centering\arraybackslash$} p{1.3cm} <{$} >{\centering\arraybackslash$} p{1.3cm} <{$} >{\centering\arraybackslash$} p{1.3cm} <{$}
}{ -0.1892  & -1.008 &  0.0 \\     0.0    &     0.0  & -7.1070},\\
G^{(2)}&=\matrix{
>{\centering\arraybackslash$} p{1.3cm} <{$} >{\centering\arraybackslash$} p{1.3cm} <{$} >{\centering\arraybackslash$} p{1.3cm} <{$}
}{ -0.1892  & -1.008 & 0.0 \\  0.0  &  0.0 & 0.0},
\end{split}
\end{equation*}
\begin{equation*}
\begin{split}
G^{(3)}&=\matrix{
>{\centering\arraybackslash$} p{1.3cm} <{$} >{\centering\arraybackslash$} p{1.3cm} <{$} >{\centering\arraybackslash$} p{1.3cm} <{$}
}{ -0.1892  & -1.008 & 0.0 \\  0.0  &  0.0 & 0.0},\\
G^{(4)}&=\matrix{
>{\centering\arraybackslash$} p{1.3cm} <{$} >{\centering\arraybackslash$} p{1.3cm} <{$} >{\centering\arraybackslash$} p{1.3cm} <{$}
}{ 0.0 & 0.0 & 0.0 \\  0.0  & 0.0 &  6.6070}.
\end{split}
\end{equation*}
Figure~\ref{figure1} illustrates the closed-loop performance measure $\|T_{zw}(s;\Gamma^0,\alpha)\|_\infty$ for the initial control design strategy $\Gamma^0$ (left) and the suboptimal closed-loop performance measure $\|T_{zw}(s;\Gamma^*,\alpha)\|_\infty$ (right) as a function of the system parameters $\alpha_i$, $i=1,2$.
\hfill $\blacktriangleleft$
\end{example}

Now, we adapt the definition of the competitive ratio (see~\cite{langbort2010distributed,FLJ2012}) to our problem formulation. Using this measure, we can characterize the value of the model parameter information in the control design. Assume that for every $\alpha \in \mathcal{A}$, there exists an optimal controller $K^*(\alpha) \in \mathcal{K}$ such that
\begin{equation*}
J(K^*(\alpha),\alpha)\leq J(K,\alpha), \; \forall K \in \mathcal{K}.
\end{equation*}
Notice that $K^*:\mathcal{A}\rightarrow \mathcal{K}$ is not necessarily in $\mathcal{C}$ or $\mathcal{C}((\eta_{\ell'})_{\ell'=1}^{L'})$ since its entries might depend on all the parameters in the vector $\alpha$ (and not just some specific subset of them). Now, we define the competitive ratio of a control design method $\Gamma$ as
\begin{equation*}
r(\Gamma)= \sup_{\alpha \in \mathcal{A}} \frac{J(\Gamma(\alpha),\alpha)}{J(K^*(\alpha),\alpha)},
\end{equation*}
with the convention that ``$\frac{0}{0}$'' equals one. Let us calculate this ratio for our numerical example.

\setcounter{example}{0}
\begin{example}[Cont'd] For the definition of the competitive ratio, we need to calculate $K^*(\alpha)$. To do so, we assume that the control graph $\mathcal{G}_\mathcal{K}$ is a complete graph. Consider the output vectors
$$
y_1(t)=y_2(t)=\matrix{c}{ x_1(t) \\ x_2(t) } \in\mathbb{R}^2.
$$
Hence, we are dealing with full state feedback, but it is still a parameter-dependent control design problem. For any $\alpha\in\mathcal{A}$, $K^*(\alpha)$ is a static controller, which can be derived from a convex optimization problem~\cite{zhou1998essentials}. For this setup, let us run Algorithm~\ref{alg:2} with $\alpha(0)=[0 \;\;\; 0]^\top$ and
$$
\Gamma^0(\alpha)=\matrix{
>{\centering\arraybackslash$} p{1.3cm} <{$} >{\centering\arraybackslash$} p{1.3cm} <{$} >{\centering\arraybackslash$} p{1.3cm} <{$} >{\centering\arraybackslash$} p{1.3cm} <{$}
}{0.0 & 0.0 & 0.0 & 0.0 \\ 0.0 & 0.0 & 0.0 & -0.5}.
$$
Then, we get $\Gamma^*(\alpha)=G^{(1)}+G^{(2)}\alpha_1+G^{(3)}\alpha_1^2 +G^{(4)}\alpha_2,$
where
\begin{equation*}
\begin{split}
G^{(1)}&=\matrix{
>{\centering\arraybackslash$} p{1.3cm} <{$} >{\centering\arraybackslash$} p{1.3cm} <{$} >{\centering\arraybackslash$} p{1.3cm} <{$} >{\centering\arraybackslash$} p{1.3cm} <{$}
}{   -0.0624 &  -0.1023   &      0.0   &      0.0 \\
0.0 &     0.0 &  -0.3992  & -1.1650},
\end{split}
\end{equation*}
\begin{equation*}
\begin{split}
G^{(2)}&=\matrix{
>{\centering\arraybackslash$} p{1.3cm} <{$} >{\centering\arraybackslash$} p{1.3cm} <{$} >{\centering\arraybackslash$} p{1.3cm} <{$} >{\centering\arraybackslash$} p{1.3cm} <{$}
}{     -0.0624 &  -0.1023    &    0.0    &     0.0 \\
0.0  &       0.0   &      0.0  &       0.0},\\
G^{(3)}&=\matrix{
>{\centering\arraybackslash$} p{1.3cm} <{$} >{\centering\arraybackslash$} p{1.3cm} <{$} >{\centering\arraybackslash$} p{1.3cm} <{$} >{\centering\arraybackslash$} p{1.3cm} <{$}
}{   -0.0624  & -0.1023      &     0.0    &     0.0 \\
0.0    &     0.0     &    0.0     &    0.0},\\
G^{(4)}&=\matrix{
>{\centering\arraybackslash$} p{1.3cm} <{$} >{\centering\arraybackslash$} p{1.3cm} <{$} >{\centering\arraybackslash$} p{1.3cm} <{$} >{\centering\arraybackslash$} p{1.3cm} <{$}
}{ 0.0  &       0.0    &     0.0   &      0.0 \\
0.0      &   0.0  &  0.3992  &  0.6650}.
\end{split}
\end{equation*}
To calculate the competitive ratio, we grid the set of all eligible parameters $\mathcal{A}$ and calculate $K^*$ (and its closed-loop performance) for each grid point. This results in
$$
r(\Gamma^*)=\sup_{\alpha \in \mathcal{A}} \frac{J(\Gamma(\alpha),\alpha)}{J(K^*(\alpha),\alpha)}=1.1475.
$$
Hence, the closed-loop performance of $\Gamma^*$ can be at most 15\% worse than the performance of the control design strategy with access to the full parameter vector. We can also infer that, although using gradient descent optimization, $\Gamma^*$ is close to the global solution of the optimization problem~\eqref{eqn:optimization} since the performance cost of the global solution must lay somewhere between the performances of $\Gamma^*$ and $K^*$, which are very close to each other thanks to the relatively small $r(\Gamma^*)$. The 15\% performance degradation is partly due to using local model information, but it is also due to the use of the basis functions $\{\eta_{\ell'}\}_{\ell'=1}^4$ to expand the control design strategies (since $\spanf(\eta_{\ell'})_{\ell'=1}^4$ is not dense in $\mathcal{C}$). To portray this fact quantitatively, let us assume that the design graph $\mathcal{G}_\mathcal{C}$ is a complete graph and use Algorithm~\ref{alg:2} to calculate a saddle point $(\Gamma^\bullet,\alpha^\bullet$) of $J$. Doing so, we get
$$
r(\Gamma^\bullet)=1.1344,
$$
so about 13\% of the performance degradation is caused by the choice of the basis functions $\{\eta_{\ell'}\}_{\ell'=1}^4$. This amount can be certainly reduced by increasing $L'$ (i.e., adding to the number of basis functions employed to describe the control design strategies).
\hfill $\blacktriangleleft$
\end{example}

\begin{figure}[t!]
\centering
\includegraphics[width=.5\linewidth]{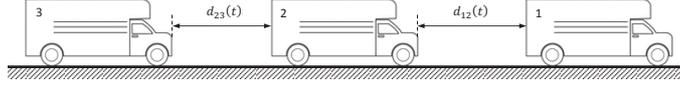}
\caption{\label{figure_platoon} Regulating the distance between three vehicles in a platoon.}
\end{figure}

\begin{figure}[t]
\centering
\begin{tikzpicture}[>=stealth',initial/.style={}]
\node[state,minimum size=0.1cm,scale=0.7](P1){$P_1$};
\node[draw=none,fill=none,node distance=.4cm](dummy)[below =of P1]{};
\node[draw=none,fill=none,node distance=0.4cm][right =of dummy] {$\mathcal{G}_\mathcal{K}$};
\node[state,minimum size=0.1cm,node distance=1.0cm,scale=0.7] (P2) [right =of P1]  {$P_2$};
\node[draw=none,fill=none,node distance=.35cm](dummy2)[left =of P2]{};
\node[state,minimum size=0.1cm,node distance=0.8cm,scale=0.7] (P3) [above =of dummy2]  {$P_3$};
\path (P1) edge [->,in=160,out=240,looseness=8,double=black] (P1);
\path (P2) edge [->,in=-60,out=+20,looseness=8,double=black] (P2);
\path (P3) edge [->,in=50,out=+130,looseness=8,double=black] (P3);
\path (P2) edge [->,relative=false,out=195,in=-15,double=black] (P1);
\path (P1) edge [->,relative=false,out=15,in=165,double=black] (P2);
\path (P2) edge [->,relative=false,out=105,in=-40,double=black] (P3);
\path (P3) edge [->,relative=false,out=-70,in=130,double=black] (P2);
\end{tikzpicture}
\vspace{-.2in}
\caption{\label{graph:example:GK2} The control graph in the vehicle platooning.} 
\end{figure}
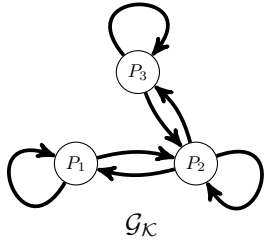

\begin{figure*}[t]
\centering
\vspace{-.1in}
$$
\hspace{-.1in}
\begin{array}{ccc}
\hspace{-.4in}
\begin{tikzpicture}[>=stealth',initial/.style={}]
\node[state,minimum size=0.1cm,scale=0.7](P1){$P_1$};
\node[draw=none,fill=none,node distance=.4cm](dummy)[below =of P1]{};
\node[draw=none,fill=none,node distance=0.4cm][right =of dummy] {$\mathcal{G}_\mathcal{C}$};
\node[state,minimum size=0.1cm,node distance=1.0cm,scale=0.7] (P2) [right =of P1]  {$P_2$};
\node[draw=none,fill=none,node distance=.35cm](dummy2)[left =of P2]{};
\node[state,minimum size=0.1cm,node distance=0.8cm,scale=0.7] (P3) [above =of dummy2]  {$P_3$};
\path (P1) edge [->,in=160,out=240,looseness=8,double=black] (P1);
\path (P2) edge [->,in=-60,out=+20,looseness=8,double=black] (P2);
\path (P3) edge [->,in=50,out=+130,looseness=8,double=black] (P3);
\end{tikzpicture}
&
\begin{tikzpicture}[>=stealth',initial/.style={}]
\node[state,minimum size=0.1cm,scale=0.7](P1){$P_1$};
\node[draw=none,fill=none,node distance=.4cm](dummy)[below =of P1]{};
\node[draw=none,fill=none,node distance=0.4cm][right =of dummy] {$\mathcal{G}'_\mathcal{C}$};
\node[state,minimum size=0.1cm,node distance=1.0cm,scale=0.7] (P2) [right =of P1]  {$P_2$};
\node[draw=none,fill=none,node distance=.35cm](dummy2)[left =of P2]{};
\node[state,minimum size=0.1cm,node distance=0.8cm,scale=0.7] (P3) [above =of dummy2]  {$P_3$};
\path (P1) edge [->,in=160,out=240,looseness=8,double=black] (P1);
\path (P2) edge [->,in=-60,out=+20,looseness=8,double=black] (P2);
\path (P3) edge [->,in=50,out=+130,looseness=8,double=black] (P3);
\path (P2) edge [->,relative=false,out=195,in=-15,double=black] (P1);
\path (P1) edge [->,relative=false,out=15,in=165,double=black] (P2);
\path (P2) edge [->,relative=false,out=105,in=-40,double=black] (P3);
\path (P3) edge [->,relative=false,out=-70,in=130,double=black] (P2);
\end{tikzpicture}
&
\begin{tikzpicture}[>=stealth',initial/.style={}]
\node[state,minimum size=0.1cm,scale=0.7](P1){$P_1$};
\node[draw=none,fill=none,node distance=.4cm](dummy)[below =of P1]{};
\node[draw=none,fill=none,node distance=0.4cm][right =of dummy] {$\mathcal{G}''_\mathcal{C}$};
\node[state,minimum size=0.1cm,node distance=1.0cm,scale=0.7] (P2) [right =of P1]  {$P_2$};
\node[draw=none,fill=none,node distance=.35cm](dummy2)[left =of P2]{};
\node[state,minimum size=0.1cm,node distance=0.8cm,scale=0.7] (P3) [above =of dummy2]  {$P_3$};
\path (P1) edge [->,in=160,out=240,looseness=8,double=black] (P1);
\path (P2) edge [->,in=-60,out=+20,looseness=8,double=black] (P2);
\path (P3) edge [->,in=50,out=+130,looseness=8,double=black] (P3);
\path (P2) edge [->,relative=false,out=195,in=-15,double=black] (P1);
\path (P1) edge [->,relative=false,out=15,in=165,double=black] (P2);
\path (P2) edge [->,relative=false,out=105,in=-40,double=black] (P3);
\path (P3) edge [->,relative=false,out=-70,in=130,double=black] (P2);
\path (P1) edge [->,relative=false,out=45,in=250,double=black] (P3);
\path (P3) edge [->,relative=false,out=220,in=65,double=black] (P1);
\end{tikzpicture}
\end{array}
\vspace{-.25in}
$$
\caption{\label{graph:example:GK1} The design graphs utilized in the vehicle platooning.}
\end{figure*}
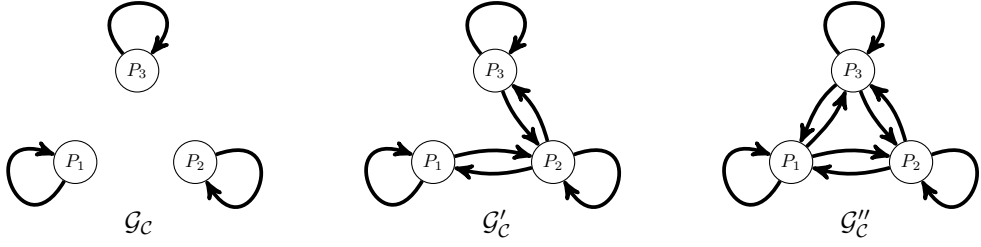

\section{Application to Vehicle Platooning}  \label{sec:platoon}
Consider a physical example where three heavy-duty vehicles are following each other closely in a platoon (see Figure~\ref{figure_platoon}). We can model this system as
$$
\dot{x}(t)=A(\alpha)x(t)+B(\alpha)u(t)+w(t),
$$
where
$$
x(t)=\matrix{ccccc}{v_1(t)&d_{12}(t)&v_2(t)&d_{23}(t)&v_3(t)}^\top \in\mathbb{R}^5,
$$
is the state vector with $v_i(t)$ denoting the velocity of vehicle $i$ and $d_{ij}(t)$ denoting the distance between vehicles $i$ and~$j$ (see Figure~\ref{figure_platoon}). Additionally, $u(t)\in\mathbb{R}^3$ is the control input (i.e., the acceleration of the vehicles), $w(t)\in\mathbb{R}^5$ is the exogenous input (i.e., the effect of wind, road quality, friction, etc),
and $\alpha=[m_1\;m_2\;m_3]^\top\in\mathbb{R}^3$ is the vector of parameters with $m_i$ denoting the mass of vehicle $i$ (scaled by its maximum allowable mass). We define the state of each subsystem as
$$
x_1(t)\hspace{-.04in}=\hspace{-.04in} \matrix{c}{v_1(t)\\ d_{12}(t)}\hspace{-.04in}, \;\;\;
x_2(t)\hspace{-.04in}=\hspace{-.04in} v_2(t),\;\;\;
x_3(t)\hspace{-.04in}=\hspace{-.04in} \matrix{c}{ d_{23}(t)\\ v_3(t)}\hspace{-.04in}.
$$
Furthermore, we have
$$
A(\alpha)=\matrix{ccccc}{-\varrho_1/m_1 & 0 & 0 & 0 & 0 \\
1 & 0 & -1 & 0 & 0 \\ 0 & 0 & -\varrho_2/m_2 & 0 & 0 \\ 0 & 0 & 1 & 0 & -1 \\ 0 & 0 & 0 & 0 & -\varrho_3/m_3}\hspace{-.05in},
$$
and
$$
B(\alpha)=\matrix{ccc}{b_1/m_1 & 0 & 0 \\ 0 & 0 & 0 \\ 0 & b_2/m_2 & 0 \\ 0 & 0 & 0 \\ 0 & 0 & b_3/m_3}\hspace{-.05in},
$$
where $\varrho_i$ is the viscous drag coefficient of vehicle $i$ and $b_i$ is the power conversion quality coefficient. These parameters are all scaled by the maximum allowable mass of each vehicle. Let us fix $\varrho_i=0.1$ and $b_i=1$ for all $i=1,2,3$. We assume that
$$
\mathcal{A}=\{\alpha\in\mathbb{R}^3\,|\, \alpha_i\in[0.5,1.0] \mbox{ for all } i=1,2,3\}.
$$
Clearly, we can satisfy Assumption~\ref{assum:1} with the choice of basis functions $\xi_1(\alpha)=1$, $\xi_2(\alpha)=1/m_1$, $\xi_3(\alpha)=1/m_2$, and $\xi_4(\alpha)=1/m_3$. Now, we assume that each vehicle only has access to the state measurements of its neighbors. This pattern is captured by the control graph $\mathcal{G}_\mathcal{K}$ in Figure~\ref{graph:example:GK2}. Hence, we get
$$
y_1(t)\hspace{-.04in}=\hspace{-.04in} \matrix{c}{\hspace{-.04in}v_1(t)\hspace{-.04in}\\ \hspace{-.04in}d_{12}(t)\hspace{-.04in} \\ \hspace{-.04in}v_2(t)\hspace{-.04in}}\hspace{-.04in}, \;
y_2(t)\hspace{-.04in}=\hspace{-.04in} \matrix{c}{\hspace{-.04in}v_1(t)\hspace{-.04in} \\ \hspace{-.04in}d_{12}(t)\hspace{-.04in} \\ \hspace{-.04in}v_2(t)\hspace{-.04in} \\ \hspace{-.04in}d_{23}(t)\hspace{-.04in} \\ \hspace{-.04in}v_3(t)\hspace{-.04in}}\hspace{-.04in}, \;
y_3(t)\hspace{-.04in}=\hspace{-.04in} \matrix{c}{\hspace{-.04in}v_2(t)\hspace{-.04in} \\ \hspace{-.04in}d_{23}(t)\hspace{-.04in} \\ \hspace{-.04in}v_3(t)\hspace{-.04in}}\hspace{-.04in},
$$
Notice that the choice of these particular observation vectors is convenient as the vehicles can measure them directly (using velocity and distance sensors mounted on the front and the back of the vehicles) and they do not need to relay these measurements to each other through a communication medium. For safety reasons, we would like to ensure that the exogenous inputs do not significantly influence the distances between the vehicles. However, we would like to guarantee this fact using as little control action as possible. We capture this goal by minimizing the $H_\infty$-norm of the closed-loop transfer function from the exogenous inputs $w(t)$ to
$$
z(t)=\matrix{ccccc}{d_{12}(t)&d_{23}(t)&u_1(t)&u_2(t)&u_3(t)}^\top.
$$
Let us use the basis functions $\eta_1(\alpha)=1$, $\eta_2(\alpha)=m_1$, $\eta_3(\alpha)=m_1^2$, $\eta_4(\alpha)=m_2$, $\eta_5(\alpha)=m_2^2$, $\eta_6(\alpha)=m_3$, and $\eta_7(\alpha)=m_3^2$ to expand the control design strategies. We use Algorithm~\ref{alg:2} to compute the optimal control design strategy. Notice that the open-loop system has two poles on the imaginary axis for all $\alpha\in\mathcal{A}$. To eliminate this problem, we initialize the algorithm with an stabilizing control design strategy
\vspace{-0.07in}
$$
\Gamma^0(\alpha)\hspace{-.03in}=\hspace{-.03in}\matrix{ccccccccccc}{
\hspace{-.05in}-3 & 0 & 0 & 0 & \hspace{-.05in}0  &  \hspace{-.05in}0 & 0  & 0 & 0 & 0  & 0 \hspace{-.05in}\\
\hspace{-.05in} 0 & 0 & 0 & 0 & \hspace{-.05in}15 & \hspace{-.05in}-5 & 10 & 0 & 0 & 0  & 0 \hspace{-.05in}\\
\hspace{-.05in} 0 & 0 & 0 & 0 & \hspace{-.05in}0  &  \hspace{-.05in}0 & 0  & 0 & 0 & 10 & -5\hspace{-.05in}}\hspace{-.05in}.
$$
We pick $\alpha(0)=[0.5\;0.5\;0.5]^\top$, $\varepsilon=10^{-2}$, $\epsilon=10^{-3}$, and $\mu_k=1/k$ for all $k\in\mathbb{Z}_{\geq 1}$. For comparisons, note that
$$
\max_{\alpha\in\mathcal{A}} \;\; \left\|T_{zw}\left(s;\Gamma^0,\alpha\right)\right\|_\infty=11.9626.
$$
In the following subsections, we calculate optimal control design strategy under three different information regimes. Note that the importance of communicating parameter information for vehicle platooning was also considered in~\cite{6160938}, where the authors designed decentralized linear quadratic controllers.

\subsection{Local Model Information Availability}
We start with the case where each local controller only relies on the mass of its own vehicle. This model information availability corresponds to the design graph $\mathcal{G}_\mathcal{C}$ in Figure~\ref{graph:example:GK1}. For this case, we get the performance
$$
\max_{\alpha\in\mathcal{A}} \;\; \left\|T_{zw}\left(s;\Gamma^{\mathrm{local}},\alpha\right) \right\|_\infty=4.7905,
$$
where $\Gamma^{\mathrm{local}}$ is the outcome of Algorithm~\ref{alg:2} with the described initialization.

\subsection{Limited Model Information Availability}
Here, we let the neighboring vehicles communicate their mass to each other. This model information availability corresponds to the design graph $\mathcal{G}'_\mathcal{C}$ in Figure~\ref{graph:example:GK1}. For this information regime, we get
$$
\max_{\alpha\in\mathcal{A}} \;\; \left\|T_{zw}\left(s;\Gamma^{\mathrm{limited}},\alpha\right) \right\|_\infty=3.5533,
$$
where $\Gamma^{\mathrm{limited}}$ is the outcome of Algorithm~\ref{alg:2}. Clearly, we get a 25\% improvement in comparison to $\Gamma^{\mathrm{local}}$.

\subsection{Full Model Information Availability}
Finally, we consider the case where each local controller has access to all the model parameters (i.e., the mass of all other vehicles). This model information availability corresponds to the design graph $\mathcal{G}''_\mathcal{C}$ in Figure~\ref{graph:example:GK1}. We get
$$
\max_{\alpha\in\mathcal{A}} \;\; \left\|T_{zw}\left(s;\Gamma^{\mathrm{full}},\alpha\right) \right\|_\infty=3.3596,
$$
where $\Gamma^{\mathrm{full}}$ is the outcome of Algorithm~\ref{alg:2}. It is interesting to note that with access to full model information, we only improve the closed-loop performance by another 5\% in comparison to $\Gamma^{\mathrm{limited}}$. This might be caused by the fact that the first and the third vehicles are not directly interacting. 

\section{Conclusions} \label{sec:conclusion}
In this paper, we studied optimal static control design under limited model information and partial state measurements for continuous-time linear parameter-dependent systems. We defined the control design strategies as mappings from the set of parameters to the set of controllers. Then, we expanded these mappings using basis functions. We proposed a numerical optimization method based on consecutive local minimizations and maximizations of the $H_\infty$--norm of the closed-loop transfer function with respect to the control design strategy gains and the system parameters. The optimization algorithm relied on using the subgradients of this closed-loop performance measure. As future work, we will focus on finding the best basis functions for expanding the control design strategies. We will also study the rate at which the closed-loop performance improves when increasing the number of the basis functions.

\bibliographystyle{ieeetr}
\bibliography{compile_new}

\begin{thebibliography}{10}

\bibitem{levine1971optimal}
W.~Levine, T.~Johnson, and M.~Athans, ``Optimal limited state variable feedback
  controllers for linear systems,'' {\em IEEE Transactions on Automatic
  Control}, vol.~16, no.~6, pp.~785--793, 1971.

\bibitem{witsenhausen1968counterexample}
H.~Witsenhausen, ``A counterexample in stochastic optimum control,'' {\em SIAM
  Journal on Control}, vol.~6, no.~1, pp.~131--147, 1968.

\bibitem{sandell1978survey}
N.~Sandell~Jr, P.~Varaiya, M.~Athans, and M.~Safonov, ``Survey of decentralized
  control methods for large scale systems,'' {\em IEEE Transactions on
  Automatic Control}, vol.~23, no.~2, pp.~108--128, 1978.

\bibitem{Farokhi-thesis2012}
F.~Farokhi, ``Decentralized control design with limited plant model
  information,'' KTH Royal Institute of Technology, Licentiate Thesis, 2012.
\newblock \url{http://urn.kb.se/resolve?urn=urn:nbn:se:kth:diva-63858}.

\bibitem{langbort2010distributed}
C.~Langbort and J.~Delvenne, ``Distributed design methods for linear quadratic
  control and their limitations,'' {\em IEEE Transactions on Automatic
  Control}, vol.~55, no.~9, pp.~2085--2093, 2010.

\bibitem{FLJ2012}
F.~Farokhi, C.~Langbort, and K.~H. Johansson, ``Optimal structured static
  state-feedback control design with limited model information for
  fully-actuated systems,'' {\em Automatica}, vol.~49, no.~2, pp.~326--337,
  2013.

\bibitem{farokhi2011dynamic}
F.~Farokhi and K.~H. Johansson, ``Dynamic control design based on limited model
  information,'' in {\em Proceedings of the 49th Annual Allerton Conference on
  Communication, Control, and Computing}, pp.~1576--1583, 2011.

\bibitem{Apkarian1576856}
P.~Apkarian and D.~Noll, ``Nonsmooth {$H_\infty$} synthesis,'' {\em IEEE
  Transactions on Automatic Control}, vol.~51, no.~1, pp.~71--86, 2006.

\bibitem{packard1994gain}
A.~Packard, ``Gain scheduling via linear fractional transformations,'' {\em
  Systems \& Control Letters}, vol.~22, no.~2, pp.~79--92, 1994.

\bibitem{rugh2000research}
W.~J. Rugh and J.~S. Shamma, ``Research on gain scheduling,'' {\em Automatica},
  vol.~36, no.~10, pp.~1401--1425, 2000.

\bibitem{leith2000survey}
D.~J. Leith and W.~E. Leithead, ``Survey of gain-scheduling analysis and
  design,'' {\em International Journal of Control}, vol.~73, no.~11,
  pp.~1001--1025, 2000.

\bibitem{nichols1993gain}
R.~A. Nichols, R.~T. Reichert, and W.~J. Rugh, ``Gain scheduling for
  {H-infinity} controllers: A flight control example,'' {\em IEEE Transactions
  on Control Systems Technology}, vol.~1, no.~2, pp.~69--79, 1993.

\bibitem{morse1996supervisory}
A.~S. Morse, ``Supervisory control of families of linear set-point controllers
  {Part I}: Exact matching,'' {\em IEEE Transactions on Automatic Control},
  vol.~41, no.~10, pp.~1413--1431, 1996.

\bibitem{scorletti1998improved}
G.~Scorletti and L.~E. Ghaoui, ``Improved {LMI} conditions for gain scheduling
  and related control problems,'' {\em International Journal of Robust and
  nonlinear control}, vol.~8, no.~10, pp.~845--877, 1998.

\bibitem{apkarian2000parameterized}
P.~Apkarian and H.~D. Tuan, ``Parameterized {LMIs} in control theory,'' {\em
  SIAM Journal on Control and Optimization}, vol.~38, no.~4, pp.~1241--1264,
  2000.

\bibitem{Johnson1099586}
T.~Johnson and M.~Athans, ``On the design of optimal constrained dynamic
  compensators for linear constant systems,'' {\em IEEE Transactions on
  Automatic Control}, vol.~15, no.~6, pp.~658--660, 1970.

\bibitem{zhou1998essentials}
K.~Zhou and J.~C. Doyle, {\em Essentials of Robust Control}.
\newblock Prentice Hall, 1998.

\bibitem{shor1985minimization}
N.~Z. Shor, {\em Minimization methods for non-differentiable functions}.
\newblock Springer-Verlag Berlin, 1985.

\bibitem{boyd2003subgradient}
S.~Boyd, L.~Xiao, and A.~Mutapcic, ``Subgradient methods,'' {\em {Lecture notes
  of EE392O, Stanford University, Autumn Quarter}}, 2003.
\newblock
  \url{http://www.stanford.edu/class/ee364b/lectures/subgrad_method_notes.pdf}.

\bibitem{6160938}
A.~Alam, A.~Gattami, and K.~Johansson, ``Suboptimal decentralized controller
  design for chain structures: Applications to vehicle formations,'' in {\em
  Prceedings of the 50th IEEE Conference on Decision and Control and European
  Control Conference}, pp.~6894--6900, 2011.

\end{thebibliography}
\end{document}